\documentclass[reqno]{amsart}
\usepackage{amsfonts,amsmath,amsthm,amssymb,latexsym, cite}
\usepackage{graphicx}
\usepackage[colorlinks,plainpages,citecolor=blue, linkcolor=blue]{hyperref}%,backref,bookmarksnumbered,

\theoremstyle{plain}
\newtheorem{thm}{Theorem}[section]
\newtheorem{defn}[thm]{Definition}

\newtheorem{prop}[thm]{Proposition}

\numberwithin{equation}{section}

\newcommand{\dmn}{\mathop{\rm dom}}
\renewcommand{\Im}{\mathop{\rm Im}}
\newcommand{\supp}{\mathop{\rm supp}}
\newcommand{\sgn}{\mathop{\rm sgn}}
\renewcommand{\kappa}{\varkappa}
\newcommand{\rmi}{{\rm i}}
\newcommand{\Real}{\mathbb R}
\newcommand{\eps}{\varepsilon}

\begin{document}

\title[Schr\"{o}dinger operators with  $(\alpha\delta'+\beta \delta)$-like potentials]
{Schr\"{o}dinger operators with {\Large $(\alpha\delta'+\beta \delta)$}-like potentials: norm resolvent convergence and solvable models}

\author{Yuriy Golovaty}%
\address{Department of Mechanics and Mathematics,
  Ivan Franko National University of Lviv\\
  1 Universytetska str., 79000 Lviv, Ukraine}
\curraddr{}
\email{yu\_\,holovaty@franko.lviv.ua}


\begin{abstract}
 For real functions $\Phi$ and $\Psi$ that are integrable and compactly supported, we prove the norm resolvent convergence, as $\eps\to0$, of a family $S_\eps$ of one-dimensional Schr\"odinger operators on the line of the form
 $$
    S_\eps= -\frac{d^2}{d x^2}+\alpha\eps^{-2}\Phi(\eps^{-1}x)+\beta\eps^{-1}\Psi(\eps^{-1}x).
 $$
The limit results are shape-dependent: regardless of the convergence of potentials in the sense of distributions the limit operator $S_0$ exists and strongly depends on the pair $(\Phi,\Psi)$.
We show that it is impossible to assign just one self-adjoint ope\-ra\-tor to the pseudo-Hamiltonian $-\frac{d^2}{dx^2}+\alpha\delta'(x)+\beta\delta(x)$, which is a symbolic notation only for a wide variety
of quantum systems with  quite different properties.
\end{abstract}

\maketitle

%\tableofcontents
\section{Introduction}

The Schr\"{o}dinger operators with singular potentials supported on a discrete set (such potentials are usually termed ``point interactions'') have attracted considerable attention both in the physical and mathematical literature  from the early thirties of the last century.  To understand the nature of quantum systems it appeared conceivable to analyze their general features about interactions with a range much smaller then the atomic size. Historically point interactions were introduced in quantum mechanics as  limits of families of squeezed potentials.
The quantum mechanical models that are based on the concept of zero range quantum interactions
reveal an undoubted effectiveness whenever solvability together with non triviality is required.
General references for this fascinating area are \cite{Albeverio2edition, AlbeverioKurasov},
which provide extensive documentation of pertinent material.

In spite of all advantages of the solvable models, which are widely used in
various applications to quantum physics, they  give rise to many mathematical difficulties.
One of the main difficulty of the analysis of zero-range interactions,
compared to Schr\"{o}dinger operators with short-range potentials, is that the Schr\"{o}dinger operators with singular potentials are often only formal differential expressions, and for the corresponding differential equations    no solution exists even in the sense of distributions.
In 1961 Berezin and Faddeev \cite{BerezinFadeev}    suggested how such formal Schr\"{o}dinger operators can be constructed as mathematically well-defined objects, and for the first time  a formal Hamiltonian was written as a self-adjoint operator derived by the theory of self-adjoint extensions of symmetric operators (see  \cite{GorbachuksBook:1991} for more details).

There exists a large body of results  on this subject. It is impossible to refer to all relevant papers, and I confine myself to a  brief overview of the one-dimensional case.
In the recent years, many results were obtained for the point interactions based on
the theories of self-adjoint extensions of symmetric operators, singular  quadratic  forms,  boundary triples and almost solvable extensions
\cite{KochubeyUMZh1989,KochubeySibMatZh1991,%
KoshmanenkoBook:1999,KostenkoMalamud:2009,%
AlbeverioKoshmanenkoKurasovNizhnik:2002}, and once more we refer to \cite{Albeverio2edition, AlbeverioKurasov} for a exhaustive list of references therein.
Another way to define a formal Schr\"{o}dinger operator with a distributional potential $v\in \mathcal{D}'(\Real)$ is to approximate it by  Schr\"{o}dinger  operators with more smooth potentials $v_\eps$ obtained by a suitable regularization as well as to use the concept of quasi-derivatives
\cite{AlbeverioNizhnik:2000,HrynivMykytyuk:2003,ExnerNeidhardtZagrebnovCMP,GoriunovMikhailetsMFAT:2010,%
 GoriunovMikhailetsMN:2010,ShkalikovSavchukMN1999, ShkalikovSavchukTMMO2003, SebaHalfLine:1985, SebRMP:1986}.
In \cite{Heydarov:2005, GadellaNegroNietoPL2009, KurasovJMAA:1996, NizhFAA2006} the solvable models for 1D Schr\"{o}dinger operators were based on specific products of $\delta^{(k)}$ and discontinuous functions, where $\delta$ is the Dirac delta function. It is worth to note that the difficulty of dealing with the multiplication in $\mathcal{D}'$ may also be overcome by using the new algebras of  generalized  functions \cite{Antonevich2:1999}.

It is common knowledge that all nontrivial point interactions at a point $x$ can be described by the coupling conditions
\begin{equation}\label{PIConditions}
\begin{pmatrix} \psi(x+0) \\ \psi'(x+0) \end{pmatrix}
=C\begin{pmatrix} \psi(x-0) \\ \psi'(x-0)\end{pmatrix},
\qquad\quad C=e^{i\varphi}\begin{pmatrix} c_{11} & c_{12} \\ c_{21} & c_{22} \end{pmatrix},
\end{equation}
where $\varphi\in [-\frac{\pi}{2},\frac{\pi}{2}]$, $c_{kl}\in\mathbb{R}$ and $c_{11}c_{22}-c_{12}c_{21}=1$.
For the physically based classification of these interactions we refer the reader to the recent preprint \cite{BrascheNizhnik:2011}, where in particular  one singles out the four most important cases:
$\delta$-potentials, $\delta'$-interactions, $\delta'$-potentials and $\delta$-magnetic potentials.
For a quantum system described by the Schr\"{o}dinger operator with a smooth enough  potential localized in a neighbourhood of $x$ one can often assign a point interaction with some matrix in \eqref{PIConditions} so that the corresponding zero-range model governs the  quantum dynamics of the true interaction with admissible fidelity, especially for low-energy particles.
However the connection between real short-range and idealized point interactions is very complex and ambiguously determined. This is  certainly the reason why  there is a number of papers on this subject  occasionally even with conflicting conclusions.

As for the $\delta$-potential, any smooth approximation of $\beta \delta(x)$ leads to the same solvable model given by conditions \eqref{PIConditions} with the matrix
\begin{equation}\label{DeltaMatrix}
    C=\begin{pmatrix} 1 & 0 \\ \beta & 1 \end{pmatrix}.
\end{equation}
Thus the result is shape-independent, i.e., it is not sensitive to  regularization.
In this case the limiting argument admits a straightforward interpretation.
The nonzero off-diagonal element of $C$ implicitly involves
the integral of the approxima\-ting potential, and hence a slow particle on the line ``feels''
only the average value of a localized potential.

The situation changes if we turn to the $\delta'$-potential.
The usual regularization of $\delta'(x)$ is a sequence  $\varepsilon^{-2}v(\varepsilon^{-1}x)$
with a zero-mean function  $v\in C^\infty_0(\Real)$.
It was shown in  \cite{GolovatyHryniv:2010} that for \textit{almost all} functions $v$ the best zero-range appro\-ximation to the Hamilto\-nian $H_\eps=-\frac{d^2}{dx^2}+\varepsilon^{-2}v(\varepsilon^{-1}x)$  is the free Hamiltonian $-\frac{d^2}{dx^2}$ subject to the split boundary conditions $\psi(-0)=\psi(+0)=0$.
These conditions define  a non-transparent interaction at the origin. However,  there exist  so-called \textit{resonant potentials} $v$ (see below for the precise definition) for which the limit behaviour of quantum system can be characterized  by
the nontrivial point interaction with the coupling matrix
\begin{equation}\label{DeltaPrimeMatrix}
    C=\begin{pmatrix} \theta & 0 \\ 0 & \theta^{-1} \end{pmatrix},
\end{equation}
where $\theta=\theta(v)$ is a spectral characteristics of the potential $v$.
Incidentally, it is of interest that for \textit{any} shape $v$ there exists a countable set of so-called resonant coupling constants $\alpha_k$ for which the spreading potentials $\alpha_kv$ are resonant. These results were recently extended to potentials $v$ of the Faddeev-Marchenko class \cite{GolovatyHryniv:2011} and generalized to the case of quantum graphs \cite{MankoJPA:2010}.

Therefore the results on $\delta'$-potentials become shape-dependent:
depending on  $v$ the Hamiltonians $H_\eps$ regularize different kinds of point interactions, nevertheless all of them involve the $\delta'$-like potentials. Hence, it is impossible to assign just one self-adjoint ope\-ra\-tor to the pseudo-Hamiltonian $-\frac{d^2}{dx^2}+\alpha\delta'(x)$, which is a symbolic notation  for a wide variety
of quantum systems with  quite different qualitative and quantitative cha\-rac\-teristics.

It has been known for a very long time that the $\delta'$-potential defined through the regularization $\varepsilon^{-2}v(\varepsilon^{-1}x)$ is opaque acting as a perfect wall, see widely cited \v{S}eba's paper \cite{SebRMP:1986} of 1986. However, such a conclusion is in contradiction with the analysis of $H_\eps$  with piece-wise constant potentials $v$  performed recently by Zolotaryuk a.o.
\cite{ChristianZolotarIermak03,Zolotaryuk:2008,Zolotaryuk:2010,Zolotaryuks:2011},
where the resonances in the transmission probability for the scattering problem are established. In \cite{GolovatyManko:2009} a similar resonance phenomenon  is also obtained in the asymptotics of eigenvalues for  the Schr\"{o}dinger operators perturbed by $\delta'$-like potentials.
The authors of~\cite{AlbeverioCacciapuotiFinco:2007,CacciapuotiExner:2007} faced the question on the convergence of $H_\eps$ in  approximation of a smooth planar quantum waveguide with a quantum graph. Under the assumption that the mean value of $v$ is different from zero, they also singled out the set of resonant potentials~$v$ producing a non-trivial limit of $H_\eps$ in the norm resolvent sense.
The situation with these controversial results was clarified in \cite{GolovatyHryniv:2010}. Curiously enough, P.~\v{S}eba was the first to discover in 1985 \cite{SebaHalfLine:1985} the resonant potentials for a similar family of the Dirichlet Schr\"{o}dinger operators
on the half-line  producing in the limit the Robin boundary condition.

This paper can be viewed as a natural continuation of  the recent work \cite{GolovatyManko:2009,GolovatyHryniv:2010,GolovatyHryniv:2011}
on the Schr\"{o}dinger operators with $\delta'$-like potentials to the case in which the potentials are a smooth enough regularization of the distribution $\alpha\delta'(x)+\beta \delta(x)$. Clearly it is to be expected that the limit results concerning such families of squeezed potentials will be also shape-dependent.

\medskip

\paragraph{\it Notation}
Throughout the paper, $W_2^l(\omega)$ stands for the Sobolev space of functions defined on a set $\omega\subset\Real$ that belong to $L_2(\omega)$
together with their derivatives up to order $l$. The norm in $W_2^2(\omega)$ is given by
$\|f\|_{W_2^2(\Omega)}= \bigl(\|f''\|^2_{L_2(\Omega)} + \|f\|^2_{L_2(\Omega)} \bigr)^{1/2}$,
where $\|f\|_{L_2(\Omega)}$ is the usual $L_2$-norm. We shall write $\|f\|$ instead of $\|f\|_{L_2(\Real)}$.

\section{Statement of Problem and Main Result}
Let us consider the formal Hamiltonian
\begin{equation*}
H=-\frac{d^2}{dx^2}+\alpha\delta'(x)+\beta \delta(x),\qquad x\in\mathbb R,
\end{equation*}
where  $\delta$ is  the Dirac delta function.
In accordance with the classic theory of distributions we have
$
(H \phi)(x)=-\phi''(x)+\alpha \phi(0)\delta'(x)+(\beta \phi(0)-\alpha \phi'(0))\delta(x)
$
for any $\phi\in C^1(\Real)$. However,  there exist no solutions in $\mathcal{D}'(\mathbb{R})$ to the equation
$H \phi=\lambda \phi$ for a nonzero $\alpha$, except for a trivial one.
The reason for this  at first sight surprising fact lies in the theory of distributions:    $\mathcal{D}'(\mathbb{R})$ is not an algebra with respect to the ``pointwise'' multiplication.

Instead of $H$, we  consider  the family of Schr\"{o}dinger operators
\begin{equation}\label{Seps}
    S_{\eps}= -\frac{d^2}{dx^2}+\frac{\alpha}{\eps^2}\Phi\left(\frac{x}{\eps}\right)
    +\frac{\beta}{\eps}\Psi\left(\frac{x}{\eps}\right), \quad
\dmn S_{\eps}=W_2^2(\mathbb{R}).
\end{equation}
with integrable  potentials $\Phi$ and $\Psi$ of compact supports.
Here $\varepsilon$ is a small positive parameter, and the coupling constants $\alpha$ and $\beta$ are assumed to be real.
One of the questions of our primary interest in this paper is the behaviour of~$S_{\varepsilon}$ as~$\eps$ tends to zero.
The motivation for this question stems from the fact that
the potentials
$$
    V_\eps=\alpha\eps^{-2}\Phi(\eps^{-1}\,\cdot\,)+\beta\eps^{-1}\Psi(\eps^{-1}\,\cdot\,)
$$
approximate the pseudopotential $\alpha\delta'+\beta \delta$  under some assumptions on $\Phi$ and $\Psi$. Indeed, if the following conditions hold
\begin{equation}\label{DDpConds}
   \int_{\Real}\Phi\,ds=0, \qquad \int_{\Real}s\Phi\,ds=-1\quad\text{and}\quad \int_{\Real}\Psi\,ds=1,
\end{equation}
 then
    $\alpha\varepsilon^{-2}\Phi(\varepsilon^{-1}x)+\beta\varepsilon^{-1}\Psi(\varepsilon^{-1}x)\to \alpha\delta'(x)+\beta \delta(x)$
 as $\eps\to 0$ in the sense of distributions. In this case, we call $\Phi$ the shape of a $\delta'$-like sequence and $\Psi$ the shape of a $\delta$-like one.

Notwithstanding the title of paper, all results presented here concern the potentials $V_\eps$ with arbitrary $\Phi$ and $\Psi$ of compact support, and the $\alpha\delta'+\beta \delta$-like potentials are only a partial case
in our considerations. Note that if the first condition in \eqref{DDpConds} is not fulfilled,
then the potentials $V_\eps$ do not converge even in the distributional sense.
However,  surprisingly  enough, regardless of the  convergence of $V_\eps$ the limit
of~$S_\eps$ exists in the norm resolvent sense (i.e., in the sense of  uniform convergence of resolvents).
\begin{defn}[\hskip-0.5pt\cite{Klaus:1982}]
    We say that the Schr\"odinger operator~$-\frac{d^2}{d s^2}+q$ in $L_2(\Real)$ possess the \emph{half-bound state} (or \emph{zero-energy resonance}) provided there exists a solution~$u$ to the equation $- u'' +qu= 0$ in $\Real$ that is bounded on the whole line, i.e. $u\in L^\infty(\Real)$. The potential $q$ is then called \emph{resonant}.
\end{defn}
Such a solution~$u$ is  unique up to a scalar factor and has nonzero limits $u(\pm\infty)$.
Our main result reads as follows.
\medskip

\textsc{Main result.} {\it
Let $\Phi$ and $\Psi$ be integrable and bounded real functions of compact support. Then the operator family $S_{\varepsilon}$ given by \eqref{Seps} converges as $\eps\to0$ in the norm resolvent sense.

If the potential $\alpha\Phi$ is resonant with a half-bound state $u_\alpha$, and $u_\alpha^\pm=u_\alpha(\pm\infty)$, then the limit operator $S_0$ is a perturbation of the free Schr\"odinger operator defined by $S_0\phi=-\phi''$
on functions $\phi$ in~$W_2^2(\Real\setminus\{0\})$ obeying the  boundary conditions at the origin
\begin{equation}\label{ResZeroRangeConds}
    \phi(+0) - \theta_\alpha \phi(-0)=0,
        \quad \phi'(+0) - \theta_\alpha^{-1} \phi'(-0)=\beta\,\kappa_\alpha \phi(-0).
\end{equation}
The parameters $\theta_\alpha$ and $\kappa_\alpha$ are specified by the potentials $\Phi$ and $\Psi$:
\begin{equation*}
    \theta_\alpha=\frac{u_\alpha^+}{u_\alpha^-}, \qquad
    \kappa_\alpha=\frac{1}{u_\alpha^-u_\alpha^+}\int_{\Real}\Psi u_\alpha^2\,dt.
\end{equation*}

Otherwise, in the non-resonant case, the limit $S_0$ is equal to the direct sum $S_-\oplus S_+$ of the Dirichlet half-line Schr\"odinger operators~$S_\pm$.
}
\medskip

The result is proved in Theorems~\ref{ThmConvergenceAtResonance} and \ref{ThmConvergenceNonResonant} below. In addition, convergence of the scattering data for $S_\eps$ to the ones for $S_0$ is established in Theorem~\ref{ThmConvScattData}.

In the resonant case, the point interaction generated by the coupling matrix
\begin{equation*}
    C(\Phi,\Psi)=\begin{pmatrix} \theta_\alpha & 0 \\ \beta \kappa_\alpha & \theta_\alpha^{-1} \end{pmatrix}
\end{equation*}
in \eqref{PIConditions} may be regarded as a first approximation to the real interaction governed by the Hamiltonian $S_\eps$. The explicit relation between $\theta_\alpha$, $\kappa_\alpha$ and the potentials $\Phi$, $\Psi$ makes it possible to carry out a quantitative analysis of the quantum system, for instance, to compute approximate values of the scattering data for given $\Phi$ and $\Psi$.

It is appropriate to mention here that in  \cite{AlbeverioDabrowskiKurasov:1998, KurasovJMAA:1996} the pseudo-potential $\alpha\delta'+\beta \delta$ was interpreted as a point interaction with the matrix
\begin{equation*}
    C=\begin{pmatrix} \frac{2+\alpha}{2-\alpha} & 0 \\ \frac{4\beta}{(2-\alpha)^2} & \frac{2-\alpha}{2+\alpha} \end{pmatrix},
\end{equation*}
and some split boundary conditions were associated with the singular values $\alpha=\pm 2$.
The solvable model was derived from the assumption that the following product formulae
\begin{equation*}
    v(x)\delta(x)=\{v\}_0\,\delta(x),\qquad v(x)\delta'(x)=\{v\}_0\,\delta'(x)-\{v'\}_0\,\delta(x)
\end{equation*}
hold, where $ \{f\}_0=\frac12(f(-0)+f(+0))$ is the mean value of a discontinuous function $f$ at $x=0$.
The spectrum and scattering  properties of this model were described in \cite{Heydarov:2005,GadellaNegroNietoPL2009}.

\section{Resonant Sets and Maps}
Since the potential $V_\eps$ has compact support shrinking to the origin, there is no loss of generality in supposing that the supports both of $\Phi$ and $\Psi$ are subsets of the interval $I=[-1,1]$.
Denote by $\mathcal{P}$  the class of real integrable and bounded functions of compact support contained in $I$.

\begin{defn}
    The \textit{resonant set} $\Lambda_\Phi$ of  potential $\Phi\in \mathcal{P}$ is the set of all real value $\alpha$ for which the operator $-\frac{d^2}{d s^2}+\alpha \Phi$ in $L_2(\Real)$ possesses a half-bound state.
\end{defn}

Suppose that a potential $q\in \mathcal{P}$ is resonant, i.e. $q$ possesses a half-bound state $u$. Then $u$, as a solution to the equation $-u''+qu=0$, is constant for $|s|>1$, because  $q$ is a zero function outside $I$.
Moreover, the restriction of $u$ to $I$ is a nontrivial solution of the problem $- u'' +qu= 0$, $s\in (-1,1)$ and $u'(-1)=0$, $u'(1)=0$. Hence, the potential $q$ is  resonant if and only if zero is an eigenvalue of the operator  $N=-\frac{d^2}{d s^2}+q$ in $L_2(I)$ subject to the Neumann boundary conditions at $s=\pm 1$.

Consequently the resonant set $\Lambda_\Phi$ coincides with the set of  eigenvalues of the problem
\begin{equation}\label{NeumanProblemWithAlpha}
     - u'' +\alpha \Phi u= 0, \quad s\in (-1,1),\qquad u'(-1)=0, \quad u'(1)=0
\end{equation}
with respect to the spectral parameter $\alpha$. In the case of a positive $\Phi$ it is clear that $\Lambda_\Phi$
is a countable subset of $\Real_+$ without finite accumulation points, and all eigenvalues of \eqref{NeumanProblemWithAlpha} are simple.
Otherwise, \eqref{NeumanProblemWithAlpha} is a  problem with indefinite  weight function \cite{CurgusLangerJDE:1989}.

Assume that $\Phi$ has only isolated turning points in $I$. This case  was considered in \cite{GolovatyManko:2009}. Then problem \eqref{NeumanProblemWithAlpha} can be associated with an operator in an appropriate Krein space.
Let $\mathcal{K}_\Phi$ be the weight $L_2$-space with the scalar product
$(f,g)=\int_{-1}^1f\overline{g}\,|\Phi|ds$ and the indefinite inner product  $[f,g]=(Jf,g)$, where  $Jf=\sgn\Phi \cdot f$.
The operator $J$ is called the fundamental symmetry.
We can introduce in $\mathcal{K}_\Phi$ the operator
\begin{multline*}
T=-\frac{1}{\Phi(s)}\frac{d^2}{ds^2},\quad \dmn T=\bigl\{g\in \mathcal{K}_\Phi\colon g\in
W_2^2(I),\\ \Phi^{-1}g''\in \mathcal{K}_\Phi,\; g'(-1)=0,\:g'(1)=0\bigr\}
\end{multline*}
that is  $J$-self-adjoint and $J$-nonnegative. The spectrum of  $T$  is real and discrete, and has the two accumulation points $-\infty$ and $+\infty$.  All nonzero eigenvalues are simple, and $\alpha=0$ is semi-simple, generically. The reader is referred to \cite{IohvidovKreinLanger} for the details of the theory.
Obviously, $\Lambda_\Phi=-\sigma(T)$. Hence, the resonant set $\Lambda_\Phi$ is discrete and unbounded in both directions.

Now suppose that the support of $\Phi$ is a disconnected subset of $I$. For the sake of
simplicity, assume that it has only one gap: $\supp \Phi=[-1,s_1]\cup[s_2,1]$ and  $s_1<s_2$.
Each solution $u$ of \eqref{NeumanProblemWithAlpha} is then a linear function on $[s_1,s_2]$. Therefore
$u'(s_2)=u'(s_1)$ and $u(s_2)-u(s_1)=l u'(s_1)$, where $l=s_2-s_1$ is the length of gap. Let us move the interval $[-1,s_1]$ up to $[s_2,1]$, identify the points $s_1$ and $s_2$, and thereafter rewrite  problem \eqref{NeumanProblemWithAlpha} as
\begin{gather}\label{GluedProbl}
\begin{gathered}
     - v'' +\alpha \Upsilon v= 0, \quad x\in (l-1,s_2)\cup(s_2,1),\quad
     v'(l-1)=0, \quad v'(1)=0,\\  v'(s_2+0)=v'(s_2-0),\quad v(s_2+0)-v(s_2-0)=l v'(s_2).
\end{gathered}
\end{gather}
The new ``glued'' potential $\Upsilon$ coincides with $\Phi$ on $[s_2,1]$,  $\Upsilon(s)=\Phi(s-l)$ for $s\in [l-1,s_2)$; hence $\Upsilon$ has a unique turning point $s=s_2$. The relation between solutions of problems \eqref{NeumanProblemWithAlpha} and \eqref{GluedProbl} is obviously given by
\begin{equation*}
    u(s)=
    \begin{cases}
        v(s+l)&\text{for } s\in [-1,s_1),\\
        v'(s_2)(s-s_1)+v(s_2-0)&\text{for } s\in [s_1,s_2],\\
        v(s)&\text{for } s\in (s_2,1].
    \end{cases}
\end{equation*}
As in the previous case, we can now construct an $J$-self-adjoint and $J$-nonnegative operator in $\mathcal{K}_\Upsilon$ associated with problem \eqref{GluedProbl} and derive the same properties of the resonant set $\Lambda_\Phi$. The similar considerations can be applied to  $\Phi$ with several gaps on its support.

In conclusion of this section, we introduce two characteristics of the potentials $\Phi$ and $\Psi$, which will turn out to be important for us later. Let us denote by $u_\alpha$ the half-bound state that corresponds to resonant potential $\alpha\Phi$. Clearly, $u_\alpha(\pm\infty)=u_\alpha(\pm 1)$.
Let $\theta$ be the map of $\Lambda_\Phi$ to $\Real$ such that
\begin{equation}\label{MapTheta}
\theta(\alpha)=\frac{u_\alpha(+1)}{u_\alpha(-1)}
\end{equation}
for all $\alpha\in \Lambda_\Phi$.  Next, let the map $\kappa$ is given by
\begin{equation}\label{MapKappa}
     \kappa(\alpha)=\frac{1}{u_\alpha(-1)u_\alpha(+1)}\int_{\Real}\Psi u_\alpha^2\,ds,\qquad\alpha\in \Lambda_\Phi.
\end{equation}
The value $\kappa(\alpha)$ describes the interaction of potentials $\Phi$ and $\Psi$
at the resonant $\alpha$. Since the half-bound state is  unique up to a scalar factor, both maps are well defined.

\begin{defn}
    We call $\theta\colon \Lambda_\Phi \to \Real$ the resonant map of $\Phi$, and $\kappa\colon \Lambda_\Phi \to \Real$ the intercoupling map for a pair of potentials $\Phi$ and $\Psi$.
\end{defn}

\section{Convergence of $S_\eps$ in Resonance Case}

In this section, we analyze more difficult resonant case where $\alpha$ is a point of the resonant set $\Lambda_\Phi$. Here and subsequently, $\theta_\alpha=\theta(\alpha)$ and $\kappa_\alpha=\kappa(\alpha)$,
where $\theta$ and $\kappa$ are the resonant and intercoupling maps for a given pair of $\Phi$ and $\Psi$. Let $u_\alpha$  be the  eigenfunction of \eqref{NeumanProblemWithAlpha} corresponding to $\alpha\in \Lambda_\Phi$ such that  $u_\alpha(-1) =1$.  From \eqref{MapTheta} and \eqref{MapKappa} it follows that
 $\theta_\alpha=u_\alpha(1)$ and
\begin{equation*}
    \kappa_\alpha=\theta_\alpha^{-1}\int_{-1}^1\Psi u_\alpha^2\,dt.
\end{equation*}
Denote by $S(\mu,\nu)$ the free Schr\"odinger operator on the line acting via $S(\mu,\nu)\phi=-\phi''$ on the domain
\begin{equation*}%\label{DomS(thetakappa)}
    \dmn S(\mu,\nu) = \bigl\{ \phi \in W_2^2(\Real\setminus\{0\}) \colon \phi(+0) = \mu \phi(-0),
        \: \phi'(+0) = \mu^{-1} \phi'(-0)+\nu \phi(-0) \bigr\}.
\end{equation*}
For each real $\nu$ and $\mu\neq 0$ the operator~$S(\mu,\nu)$ is  self-adjoint.

\begin{thm}\label{ThmConvergenceAtResonance}
    Assume that $\Phi, \Psi \in \mathcal{P}$ and $\alpha$ belongs to the resonant set $\Lambda_\Phi$. Then the operator family $S_\eps$ defined by \eqref{Seps} converges to   $S(\theta_\alpha, \beta\kappa_\alpha)$ as $\eps\to 0$ in the norm resolvent sense.
\end{thm}
\medskip

We have divided the proof into a sequence of propositions.
Fix an arbitrary $f\in L_2(\Real)$ and $\zeta\in \mathbb{C}$ with $\Im \zeta\neq 0$.
The basic idea of the proof is to construct a fair approximation to the function $y_\eps=(S_\eps-\zeta)^{-1}f$, uniformly for $f$ in bounded subsets of $L_2(\Real)$.

In the sequel, letters $C_j$, $c_j$ and $b_j$ denote various posi\-ti\-ve constants independent of~$\eps$ and $f$, whose values might be different in different proofs, and $\|f\|$ stands for the $L_2(\Real)$-norm of a function~$f$.
For abbreviation, we let $S_0$ stand for $S(\theta_\alpha, \beta\kappa_\alpha)$.

Set $y=(S_0-\zeta)^{-1}f$.
We  show that $y$ is a very satisfactory approximation to $y_\eps$ for $|x|>\eps$. The problem of choosing a close approximation to $y_\eps$ on the support of  potential $V_\eps$ is more subtle.
Denote by $v_\eps$ the solution of the Cauchy problem
  \begin{equation}\label{CPforVeps}
  \begin{cases}
     -v_\eps''+\alpha \Phi(s)v_\eps=\eps f(\eps s)-\beta y(-\eps)\Psi(s)u_\alpha(s),\qquad s\in(-1,1),\\
    \phantom{-} v_\eps(-1)=0, \quad v'_\eps(-1)=y'(-\eps).
  \end{cases}
  \end{equation}
Clearly, we have $v_\eps\in W^2_2(-1,1)$ for any $\eps>0$. The next proposition establishes some asymptotic properties of $v_\eps$ as $\eps\to 0$.

\begin{prop} \label{PropVeps}
The following holds for any $f\in L_2(\Real)$ and $\eps\in (0,1)$
\begin{equation*}
\|v_\eps\|_{W^2_2(-1,1)}\leq C_1 \|f\|,\qquad |v_\eps'(1)-y'(+0)|\leq C_2 \eps^{1/2}\|f\|.
\end{equation*}
\end{prop}

\begin{proof}
We first observe that $(S_0-\zeta)^{-1}$ is a bounded operator from~$L_2(\Real)$ to the domain of~$S_0$ equipped with the graph norm. The latter space is a subspace of $W_2^2(\Real\setminus \{0\})$, and therefore
    \begin{equation}\label{EstY}
    \|y\|_{W_2^2(\Real\setminus \{0\})}\leq c_1\|f\|,\qquad \|y\|_{C^1(\Real\setminus \{0\})}\leq c_2\|f\|,
\end{equation}
since $W_2^2(\Real\setminus \{0\})\subset C^1(\Real\setminus \{0\})$ by the Sobolev embedding theorem. Hence,
\begin{equation*}
    \|v_\eps\|_{W^2_2(-1,1)}\leq c_3\bigl(|y(-\eps)|+|y'(-\eps)|+\eps\|f(\eps\,\cdot\,)\|_{L_2(-1,1)}\bigr)\leq C_1\|f\|,
\end{equation*}
because  there exists a constant $c_4$  such that
\begin{equation}\label{EstF(eps)}
\|f(\eps\cdot)\|_{L_2(-1,1)}\leq c_4\eps^{-1/2} \|f\|.
\end{equation}
Next, multiplying equation \eqref{CPforVeps} by the eigenfunction $u_\alpha$ and integrating by parts yield
\begin{equation}\label{IntByPartsVepsUalpha}
     v'_\eps(1)=\theta_\alpha^{-1} y'(-\eps)+\beta \kappa_\alpha y(-\eps)-\eps \theta_\alpha^{-1}\int_{-1}^1 f(\eps s)\, u_\alpha(s)\,ds.
\end{equation}
Recall that the function $y$ satisfies the condition $y'(+0) = \theta_\alpha^{-1}y'(-0)+ \beta \kappa_\alpha y(-0)$.
Subtracting this equality from \eqref{IntByPartsVepsUalpha} we finally obtain
\begin{multline*}
     \bigl|v'_\eps(1)-y'(+0)\bigr|\leq |\theta_\alpha^{-1}|\,|y'(-\eps)-y'(-0)|
    +|\beta|\, |\kappa_\alpha|\, \bigl|y(-\eps)-y(-0)\bigr|\\+\eps |\theta_\alpha^{-1}|\,\|f(\eps \,\cdot\,)\|_{L_2(-1,1)} \|u_\alpha\|_{L_2(-1,1)}\leq C_2\eps^{1/2} \|f\|
\end{multline*}
in view of \eqref{EstF(eps)} and the following estimates
\begin{equation}\label{EstY(Eps)-Y(0)}
\bigr|y^{(k)}(\pm\eps)-y^{(k)}(\pm 0)\bigl|\leq \left|\int_0^{\pm\eps}|y^{(k+1)}(x)|\,dx\right|
    \leq c_6\eps^{1/2} \|y\|_{W_2^2(\Real\setminus \{0\})}    \leq c_7\eps^{1/2}\|f\|,
\end{equation}
holding for $k=0,1$.
\end{proof}

Let us introduce the function $w_\eps$ such that $w_\eps(x)=y(x)$ for $|x|>\eps$ and
$w_\eps(x)=y(-\eps)u_\alpha(x/\eps)+\eps v_\eps(x/\eps)$ for $|x|\leq\eps$.
By construction, $w_\eps$ belongs to $W_2^2(\Real\setminus\{\eps\})$.  Although $w_\eps$ is in general discontinuous at the point $x=\eps$, its jump and the jump of its first derivative at this point are small.
Indeed, $[w_\eps]_{x=\eps}=y(\eps)-\theta_\alpha y(-\eps)-\eps v_\eps(1)$ and $[w'_\eps]_{x=\eps}=y'(\eps)- v'_\eps(1)$,
where $[h]_{x=a}=h(a+0)-h(a-0)$ is a jump of a function $h$ at $x=a$.
Therefore, taking into account  Proposition~\ref{PropVeps},  estimates \eqref{EstY(Eps)-Y(0)} and the equality $y(+0) = \theta_\alpha y(-0)$, we see that the jumps can be bounded as
\begin{equation}\label{EstJumps}
\begin{aligned}
    &\bigl|[w_\eps]_{x=\eps}\bigr|\leq |y(\eps)-y(+0)|+|\theta_\alpha| |y(-\eps)-y(-0)|+ \eps \|v_\eps\|_{W^2_2(-1,1)}\leq b_1\eps^{1/2} \|f\|,\\
    &\bigl|[w'_\eps]_{x=\eps}\bigr|\leq |y'(\eps)-y'(+0)|+ |y'(+0)-v'_\eps(1)|\leq b_2\eps^{1/2} \|f\|.
\end{aligned}
\end{equation}

\begin{figure}[b]
\begin{center}
\includegraphics[scale=1.2]{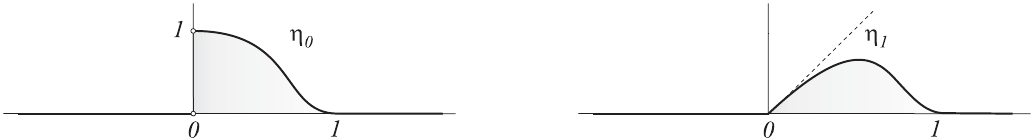}
\end{center}\caption{\label{Jump}Plots of the functions with the prescribed jumps at the origin}
\end{figure}

Let us introduce functions $\eta_0$ and $\eta_1$  that are smooth outside the origin, have compact supports contained in $[0,1]$, and have the prescribed jumps
$[\eta_0]_{x=0}=1$, $[\eta_0']_{x=0}=0$ and $[\eta_1]_{x=0}=0$, $[\eta_1']_{x=0}=1$ (see Fig.~\ref{Jump}).
Set $z_\eps(x)= [w_\eps]_{x=\eps}\,\eta_0(x-\eps)+[w'_\eps]_{x=\eps}\,\eta_1(x-\eps)$;
then in view of \eqref{EstJumps}
\begin{equation}\label{ZetaEstim}
    \max_{x\geq\eps}|z^{(k)}_\eps(x)|\leq b_3\eps^{1/2}\|f\|
\end{equation}
for some $b_3$ and $k=0,1,2$. Additionally, $z_\eps =0$ on~$(-\infty,\eps)$.

Clearly, the function
\begin{equation}\label{YepsTilda}
    \tilde{y}_\eps(x)=
    \begin{cases}
        y(x)-z_\eps(x) & \text{if } |x|>\eps,\\
        y(-\eps)u_\alpha(\frac x\eps)+\eps v_\eps(\frac x\eps) & \text{if } |x|\leq\eps
    \end{cases}
\end{equation}
is continuous on~$\Real$ along with its derivative and belongs to $\dmn S_\eps$.
\begin{prop}\label{PropYtildaEst}
    Fix $\zeta \in \mathbb{C}$ with $\Im \zeta \ne0$. Then the estimates
    \begin{equation}\label{EstYtilda}
        \|y_\eps-\tilde{y}_\eps\|\leq C_1\eps^{1/2}|\|f\|,\qquad
        \|\tilde{y}_\eps-y\|\leq C_2\eps^{1/2}\|f\|
    \end{equation}
hold for each $f\in L_2(\Real)$ and $\eps>0$, where $y_\eps=(S_\eps-\zeta)^{-1}f$ and $y=(S_0-\zeta)^{-1}f$.
\end{prop}

\begin{proof}
It is convenient now to rewrite the approximation $\tilde{y}_\eps$ in the form
\begin{equation*}\textstyle
    \tilde{y}_\eps(x)=(1-\chi_\eps(x))y(x)+y(-\eps)u_\alpha(\frac x\eps)+\eps v_\eps(\frac x\eps)-z_\eps(x),
\end{equation*}
where  $\chi_\eps$ is the characteristic function of $[-\eps,\eps]$, and
$u_\alpha$ and $v_\eps$ are extended by zero to the whole line.
Recalling the definition of  $y$, $u_\alpha$ and $v_\eps$, we deduce:
\begin{align*}
    &\textstyle(S_\eps-\zeta)\tilde{y}_\eps(x)=
    \bigl(-\frac{d^2}{dx^2}-\zeta\bigr)( y(x)-z_\eps(x))=f(x)+z''_\eps(x)+\zeta z_\eps(x)\\
    \intertext{for $|x|>\eps$, and}
&\begin{aligned}\textstyle
(S_\eps-&\zeta)\tilde{y}_\eps(x)=\textstyle
    \Bigl(-\frac{d^2}{dx^2}+\alpha\eps^{-2}\Phi(\eps^{-1}x)+\beta\eps^{-1}\Psi(\eps^{-1}x)-\zeta\Bigr)
    \Bigl( y(-\eps)u_\alpha(\frac x\eps)+\eps v_\eps(\frac x\eps)\Bigr)
\\&\textstyle
= \eps^{-2}
       y(-\eps)\Bigl\{-u''_\alpha+\alpha \Phi(\frac x\eps)u_\alpha\Bigr\}
+\eps^{-1}
      \Bigl\{ -v_\eps''+\alpha \Phi(\frac x\eps)v_\eps+\beta y(-\eps)\Psi(\frac x\eps)u_\alpha\Bigr\}
\\&\textstyle
+\beta \Psi(\frac x\eps) v_\eps(\frac x\eps)-\zeta \tilde{y}_\eps(x)
    =f(x)+\bigl(\beta \Psi(\frac x\eps)-\zeta\bigr) v_\eps(\frac x\eps)-\zeta y(-\eps)u_\alpha(\frac x\eps)
\end{aligned}
\end{align*}
for $|x|\leq\eps$.
Therefore $(S_\eps-\zeta)\tilde{y}_\eps=f+r_{\eps}$, where
\begin{equation*}
    r_\eps(x)=
    \begin{cases}
       z''_\eps(x)+\zeta z_\eps(x) & \text{if } |x|>\eps,\\
        \bigl(\beta \Psi(\frac x\eps)-\eps\zeta\bigr) v_\eps(\frac x\eps)-\zeta y(-\eps)u_\alpha(\frac x\eps) & \text{if } |x|\leq\eps.
    \end{cases}
\end{equation*}
Hence $\tilde{y}_\eps -y_\eps =(S_\eps-\zeta)^{-1}r_{\eps}$, and from this we conclude
\begin{equation*}
    \|y_\eps-\tilde{y}_\eps\|\leq \|(S_\eps-\zeta)^{-1}\|\, \|r_{\eps}\|
    \leq |\Im\zeta|^{-1}\|r_{\eps}\|.
\end{equation*}
We can now employ Proposition~\ref{PropVeps} and estimates \eqref{EstY}, \eqref{ZetaEstim} to derive the bound
\begin{multline*}
    \|r_{\eps}\|\leq
    c_1\|z_\eps''+\zeta z_\eps\|_{L_2(\eps,1+\eps)}
    +c_2\|v_\eps(\eps^{-1}\,\cdot\,)\|_{L_2(-\eps,\eps)}
    +c_3|y(-\eps)|\,\|u_\alpha(\eps^{-1}\,\cdot\,)\|_{L_2(-\eps,\eps)}
\\
    \leq c_4 \max_{x\geq\eps}(|z_\eps|+|z''_\eps|)
    +c_5 \eps^{1/2}\bigl(\|v_\eps\|_{L_2(-1,1)}+\|y\|_{C(\Real\setminus\{0\})}\|u_\alpha\|_{L_2(-1,1)}\bigr)
    \leq c_6\eps^{1/2}\|f\|.
\end{multline*}
This proves the first inequality in \eqref{EstYtilda}. Similarly,
\begin{multline*}
    \|\tilde{y}_\eps-y\|=\|y(-\eps)u_\alpha(\eps^{-1}\,\cdot\,)+\eps v_\eps(\eps^{-1}\,\cdot\,)-z_\eps
    -\chi_\eps y\|
    \leq c_7\eps^{1/2}|y(-\eps)|\|u_\alpha\|_{L_2(-1,1)}\\
    +c_8\eps^{3/2}\|v_\eps\|_{L_2(-1,1)}
    +c_9\max_{x\geq\eps}|z_\eps|
    +c_{10}\|y\|_{C(\Real\setminus\{0\})}\|\chi_\eps\|
    \leq c_{11}\eps^{1/2}\|f\|,
\end{multline*}
and so finish the proof.
\end{proof}

\begin{proof}[Proof of Theorem~\ref{ThmConvergenceAtResonance}]
For each $f\in L_2(\Real)$ and $\zeta\in \mathbb{C}$ with $\Im \zeta\neq 0$ we can construct the approximation $\tilde{y}_\eps$ to $y_\eps=(S_\eps-\zeta)^{-1}f$ given by \eqref{YepsTilda}.
As above, set $y=(S_0-\zeta)^{-1}f$. Applying Proposition~\ref{PropYtildaEst}, we discover
\begin{multline*}
    \|(S_\eps-\zeta)^{-1}f-(S_0-\zeta)^{-1}f\|=\|y_\eps-y\|\leq\|\tilde{y}_\eps-y_\eps\|+\|\tilde{y}_\eps-y\|
     \leq C \eps^{1/2}\|f\|,
\end{multline*}
which establishes the norm resolvent convergence of  $S_\eps$ to the operator $S(\theta_\alpha, \beta\kappa_\alpha)$.
\end{proof}

\section{Convergence of $S_\eps$ in Non-Resonance Case}

Now we study the non-resonant case when $\alpha$ does not belong to the resonant set $\Lambda_\Phi$.
Denote by $S_0$  the direct sum $S_-\oplus S_+$ of the unperturbed half-line Schr\"odinger operators
$S_\pm= -d^2/d x^2$ on~$\Real_\pm$ subject to the Dirichlet boundary condition at $x=0$. Hence
$\dmn S_0 = \{ y \in W_2^2(\Real\setminus\{0\}) \colon y(-0) =  y(+0)=0\}$.

\begin{thm}\label{ThmConvergenceNonResonant}
     If $\alpha\not\in \Lambda_\Phi$, then  the operator family $S_\eps$ given by \eqref{Seps} converges to  the direct sum $S_-\oplus S_+$, as $\eps\to 0$, in the norm resolvent sense.
\end{thm}

\begin{proof}
Exactly the same considerations, as in the previous section, apply here, with one important difference:
the function $y=(S_0-\zeta)f$ is small in a neighbourhood of the origin, since $y(0)=0$; and we have to change the approximation to $y_\eps=(S_\eps-\zeta)^{-1}f$ on the support of $V_\eps$. We set $\tilde{y}_\eps(x)=y(x)-z_\eps(x)$
for $|x|>\eps$ and $\tilde{y}_\eps(x)=\eps v_\eps(\frac x\eps)$ for $|x|\leq\eps$,
where $v_\eps$ is a solution to the boundary value problem
\begin{equation}\label{BVPforVeps}
-v_\eps''+\alpha \Phi(s)v_\eps=\eps f(\eps s),\; s\in(-1,1),\quad
     \phantom{-}v'_\eps(-1)=y'(-\eps), \quad v'_\eps(1)=y'(\eps),
\end{equation}
and $z_\eps(x)= \bigl(\eps v_\eps(-1)-y(-\eps)\bigr)\,\eta_0(-x-\eps)+\bigl(y(\eps)-\eps v_\eps(1)\bigr)\,\eta_0(x-\eps)$.
Note that the solution $v_\eps$ of \eqref{BVPforVeps} exists, because in the non-resonance case  the number $\alpha$ is not an eigenvalue of the corresponding homogeneous problem. Moreover,
\begin{equation*}
    \|v_\eps\|_{W^2_2(-1,1)}\leq
    c_1\bigl(|y'(-\eps)|+|y'(\eps)|+\eps \|f(\eps\,\cdot\,)\|_{L_2(-1,1)}\bigr) \leq c_2 \|f\|,
\end{equation*}
due to \eqref{EstF(eps)} and the apparent estimates
$\|y\|_{C^1(\Real\setminus \{0\})}\leq
c_3\|y\|_{W_2^2(\Real\setminus \{0\})}\leq
c_4\|f\|$.

Recalling that $y(0)=0$, we deduce from \eqref{EstY(Eps)-Y(0)} that $|y(\pm\eps)|\leq c_5\eps^{1/2}\|f\|$. Therefore the corrector $z_\eps$ can be bounded as
\begin{equation*}%\label{EstZepsNR}
     \max_{|x|\geq\eps}|z^{(k)}_\eps(x)|\leq c_6\eps^{1/2}\|f\|
\end{equation*}
for $k=0,1,2$, since
$|y(\pm\eps)-\eps v_\eps(\pm 1)|\leq |y(\pm\eps)|+\eps |v_\eps(\pm1)|\leq c_7\eps^{1/2}\|f\|$.

Next, an easy computation shows that $(S_\eps-\zeta)\tilde{y}_\eps=f+r_{\eps}$, where
\begin{equation*}
    r_\eps(x)=
    \begin{cases}
       z''_\eps(x)+\zeta z_\eps(x) & \text{if } |x|>\eps,\\
        \bigl(\beta \Psi(\frac x\eps)-\eps\zeta\bigr) v_\eps(\frac x\eps) & \text{if } |x|\leq\eps.
    \end{cases}
\end{equation*}
As in Proposition~\ref{PropYtildaEst}, from the estimates above we can derive $\|r_{\eps}\|\leq c_8\eps^{1/2}\|f\|$.
Hence,
\begin{multline*}
    \|(S_\eps-\zeta)^{-1}f-(S_0-\zeta)^{-1}f\|\leq\|\tilde{y}_\eps-y_\eps\|+\|\tilde{y}_\eps-y\|\\
    \leq \|(S_\eps-\zeta)^{-1}r_{\eps}\|+\|\eps v_\eps(\eps^{-1}\,\cdot\,)-z_\eps
    -\chi_\eps y\|
    \leq C \eps^{1/2}\|f\|,
\end{multline*}
which establishes the norm resolvent convergence of  $S_\eps$ to the operator
$S_-\oplus S_+$.
\end{proof}

\section{Scattering on $\alpha\delta'+\beta \delta$-Like Potentials}

\subsection{Scattering problem for $S_0$}
First, let us discuss stationary scattering associated with the Hamiltonians
$S(\theta_\alpha, \beta\kappa_\alpha)$ and $-\frac{d^2}{dx^2}$, which corresponds to the resonant case.
Consider the incoming monochromatic wave $e^{\rmi kx}$ with $k>0$ coming from the
left. Then the corresponding wave function has the form $\psi(x,k)=e^{\rmi kx} + R\,e^{-\rmi kx}$ for
$x<0$ and $\psi(x,k)=T\,e^{\rmi kx}$ for $x>0$.
Here $R$ and $T$ are respectively the reflection and transmission coefficients. The matching
conditions \eqref{ResZeroRangeConds} at the origin clearly yield
\begin{equation*}
\begin{pmatrix}
  T \\
  \rmi kT
\end{pmatrix}
=
\begin{pmatrix} \theta_\alpha & 0 \\ \beta \kappa_\alpha & \theta_\alpha^{-1} \end{pmatrix}
\begin{pmatrix}
  1+R \\
  \rmi k(1-R)
\end{pmatrix}.
\end{equation*}
Then one obtains the reflection and transmission coefficients from the left,
cf. \cite[Eq~23]{GadellaNegroNietoPL2009}:
\begin{equation}\label{LimitCoeffs}
    R(k,\alpha)=\frac{\rmi k(\theta_\alpha^{-1} -\theta_\alpha)+\beta \kappa_\alpha}{\rmi k(\theta_\alpha^{-1} +\theta_\alpha)-\beta \kappa_\alpha},\qquad
    T(k,\alpha)=\frac{2\rmi k}{\rmi k(\theta_\alpha^{-1} +\theta_\alpha)-\beta \kappa_\alpha},\quad \alpha\in\Lambda_\Phi.
\end{equation}
In the non-resonant case the scattering problem is trivial. The split condition $y(0)=0$ leads to
the equalities $R(k,\alpha)= -1$ and $T(k,\alpha)=0$.

\subsection{Convergence of the scattering data}
Next we investigate the stationary scattering for  $S_\eps(\alpha,\beta, \Phi,\Psi)$ and $-\frac{d^2}{dx^2}$, and prove that the scattering data converge as $\eps\to0$ to the scattering data for $S_0$ obtained above.
We look for the positive-energy solution to the equation $-\psi''+V_\eps \psi=k^2\psi$  given in the form
\begin{equation*}
\psi_\eps(x,k,\alpha)= \begin{cases}
                       e^{ikx}+R_\eps e^{-ikx} & \hbox{for $x<-\eps$,} \\
   A_\eps u_\eps(\eps^{-1}x,\alpha)+B_\eps v_\eps(\eps^{-1}x,\alpha) & \hbox{for $|x|<\eps$,} \\
                       \phantom{e^{ikx}+\,}T_\eps e^{ikx} & \hbox{for $x>\eps$.}
                     \end{cases}
\end{equation*}
Here $u_\eps=u_\eps(s,\alpha)$ and $v_\eps=v_\eps(s,\alpha)$ are solutions of the equation
\begin{equation}\label{AuxlProblem}
-w''+\alpha\Phi(s)w+\beta\eps\Psi(s)w=\eps^2k^2 w, \quad s\in(-1,1)
\end{equation}
subject to the  initial conditions
\begin{equation}\label{InitialCondUeVe}
    u_\eps(-1,\alpha)=1,\quad u'_\eps(-1,\alpha)=0\quad \text{and}\quad
    v_\eps(-1,\alpha)=0, \quad v'_\eps(-1,\alpha)=1
\end{equation}
respectively.
The coefficients $R_\eps$, $A_\eps$, $B_\eps$, and $T_\eps$ can be found from the linear system
\begin{equation*}
     \begin{pmatrix}
-e^{\rmi \eps k}& 1& 0& 0 \\
     \rmi \eps k e^{\rmi \eps k}& 0& 1&0\\
     0&u_\eps(1,\alpha)&v_\eps(1,\alpha)&-e^{\rmi \eps k}\\
     0&u'_\eps(1,\alpha)&v'_\eps(1,\alpha)&-\rmi \eps k e^{\rmi \eps k}
       \end{pmatrix}
     \begin{pmatrix} R_\eps\\A_\eps\\B_\eps\\T_\eps\end{pmatrix}=
     \begin{pmatrix} e^{-\rmi \eps k} \\\rmi \eps k e^{-\rmi \eps k}\\0\\0\end{pmatrix}
\end{equation*}
obtained by matching the solution and its first derivative at the points $x=\pm\eps$.
By Cramer's rule, we can derive
\begin{equation}\label{RepsTesp}
\begin{aligned}
&R_\eps(k,\alpha)= -e^{-2\rmi \eps k}\,
    \frac{u'_\eps(1,\alpha)-\rmi \eps k\bigl(u_\eps(1,\alpha)-v'_\eps(1,\alpha)\bigr)+\eps^2 k^2 v_\eps(1,\alpha)}
    {u'_\eps(1,\alpha)-\rmi \eps k\bigl(u_\eps(1,\alpha)+v'_\eps(1,\alpha)\bigr)-\eps^2 k^2 v_\eps(1,\alpha)},
    \\
&T_\eps(k,\alpha)=-e^{-2\rmi \eps k}\,
    \frac{2\rmi \eps k}
    {u'_\eps(1,\alpha)-\rmi \eps k\bigl(u_\eps(1,\alpha)+v'_\eps(1,\alpha)\bigr)-\eps^2 k^2 v_\eps(1,\alpha)}.
\end{aligned}
\end{equation}
Here we use the identity $u_\eps(1,\alpha)v'_\eps(1,\alpha)-u'_\eps(1,\alpha)v_\eps(1,\alpha)=1$ that follows from  \eqref{InitialCondUeVe} and the constancy in $\xi$ and $\eps$ of the Wronskian of $u_\eps$ and $v_\eps$.

\begin{thm}\label{ThmConvScattData}
  For each $k>0$  and  $\alpha\in\Real$ the scattering data
$R_\eps(k,\alpha)$ and $T_\eps(k,\alpha)$ converge  respectively to $R(k,\alpha)$
and\, $T(k,\alpha)$ as $\eps\to 0$, where the limit values are given by \eqref{LimitCoeffs} in the resonant case, and $R(k,\alpha)= -1$, $T(k,\alpha)=0$ otherwise.
\end{thm}
 \begin{proof} From the smooth dependence of a solution to the Cauchy problem on para\-me\-ters, we see that $u_\eps$ and $v_\eps$ converge in $C^1(-1,1)$ to the solutions $u$ and $v$ respectively of the equation $-w''+\alpha\Phi w=0$ subject to the  initial conditions
\begin{equation*}%\label{InitialCondUV}
    u(-1,\alpha)=1,\quad u'(-1,\alpha)=0\quad \text{and}\quad
    v(-1,\alpha)=0, \quad v'(-1,\alpha)=1.
\end{equation*}

 \textit{The non-resonant case.} Since  $\alpha$ is not a eigenvalue of problem \eqref{NeumanProblemWithAlpha}, we conclude that $u'(1,\alpha)$ is different from $0$. From \eqref{RepsTesp}, it immediately follows that
 $R_\eps(k,\alpha)= -1+O(\eps)$ and $T_\eps(k,\alpha)=O(\eps)$ as $\eps\to 0$.

\textit{The resonant case.} If $\alpha$ is resonant, then $u$ is an eigenfunction of \eqref{NeumanProblemWithAlpha}. Therefore $u'(1,\alpha)=0$ and $u(1,\alpha)=\theta_\alpha$.
Throughout the proof, $\theta$ and $\kappa$ denote the resonant and intercoupling maps for a pair $(\Phi, \Psi)$.
Let us substitute functions $u_\eps$ and $v_\eps$ into \eqref{AuxlProblem} alternately. Multiplying the derived identities by $u$ and integrating by parts yield
\begin{equation}\label{AsympUeVeAt1}
   \begin{aligned}
      &\theta_\alpha u'_\eps(1,\alpha)=\phantom{1+}\eps\beta \int_{-1}^1\Psi u_\eps u\,d\xi+
     \eps^2 k^2\int_{-1}^1u_\eps u\,d\xi,\\
     &\theta_\alpha v'_\eps(1,\alpha)=1+\eps\beta \int_{-1}^1\Psi v_\eps u\,d\xi+
     \eps^2 k^2\int_{-1}^1v_\eps u\,d\xi.
   \end{aligned}
\end{equation}
Therefore $u'_\eps(1,\alpha)=\eps\beta \kappa_\alpha+o(\eps)$ and
$v'_\eps(1,\alpha)=\theta_\alpha^{-1}+O(\eps)$ as $\eps\to 0$. Combining then these asymptotic formulae and
\eqref{RepsTesp}, we have
\begin{equation}
R_\eps(k,\alpha)= \frac{\rmi k(\theta_\alpha^{-1} -\theta_\alpha)+\beta \kappa_\alpha}{\rmi k(\theta_\alpha^{-1} +\theta_\alpha)-\beta \kappa_\alpha}+ o(1),\quad
T_\eps(k,\alpha)=\frac{2\rmi k}{\rmi k(\theta_\alpha^{-1} +\theta_\alpha)-\beta \kappa_\alpha}+ o(1),
\end{equation}
as $\eps\to 0$, which completes the proof.
\end{proof}

\subsection{Resonances in the transmission probability}

It follows from the theorem above that the probability of transmission across the barrier $V_\eps=\alpha\eps^{-2}\Phi(\eps^{-1}\cdot)+\beta\eps^{-1}\Psi(\eps^{-1}\cdot)$
is negligibly small for $\alpha\not\in \Lambda_\Phi$.
However, for the resonant coupling constants $\alpha$  this probability  remains non-zero as $\eps\to 0$, resulting in the existence of non-separated states.
The resonances of transmission, as shown in Fig.~\ref{FigT}, are sharp like one-point spikes, which spread for non-zero values as $\eps\to0$. The limit resonant values can be represented via  the maps $\theta$  and $\kappa$:
\begin{equation}\label{TransProbability}
    |T(k,\alpha)|^2=\frac{4k^2}{ k^2(\theta_\alpha^{-1} +\theta_\alpha)^2+\beta^2 \kappa_\alpha^2},\qquad
    \alpha\in \Lambda_\Phi.
\end{equation}
 Several special cases are of interest.
\begin{figure}[t]
  \includegraphics[scale=1]{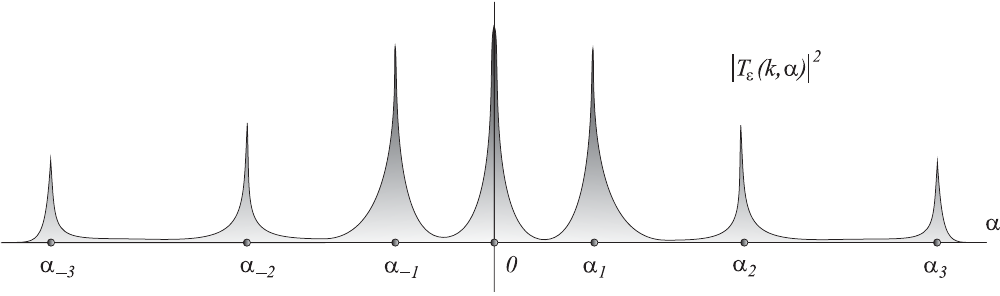}\\
  \caption{A plot of transmission probability $|T_\eps(k,\alpha)|^2$ as function of $\alpha$.}\label{FigT}
\end{figure}
\subsubsection{$\delta$-potentials} Set $\alpha=0$, then $V_\eps=\beta\varepsilon^{-1}\Psi(\varepsilon^{-1}\cdot)$. Assume also that $\int_\Real\Psi\,dt=1$. Consequently $V_\eps$ converges to $\beta\delta(x)$ in the sense of distributions. Note that the value $\alpha=0$ belongs to $\Lambda_\Phi$ for each $\Phi$.
The corresponding eigenfunction $u_0$ of  \eqref{NeumanProblemWithAlpha} is constant on $[-1,1]$, and both numbers $\theta_0$ and $\kappa_0$ are equal to $1$ (see \eqref{MapTheta}, \eqref{MapKappa}).
Hence, the limit operator $S(1, \beta)$ describes the point interaction at the origin with the coupling matrix given by \eqref{DeltaMatrix}.
In addition, from \eqref{TransProbability} we have $|T(k,0)|^2=\frac{4k^2}{ 4k^2+\beta^2}$,  which is common knowledge of the scattering by the $\beta\delta(x)$-potential.

\subsubsection{$\delta'$-like potentials} Now we set $\beta=0$. Then  $V_\eps=\alpha\varepsilon^{-2}\Phi(\varepsilon^{-1}\cdot)$. If  $\int_\Real\Phi\,dt=0$, then
$$
V_\eps(x)\to -\left(\int_\Real s\Phi(s)\,ds\right) \delta'(x)\qquad\text{in } \mathcal{D}'(\Real),
$$
and so in particular the limit can be zero. Otherwise, the family of potentials $V_\eps$ diverges in the sense of distributions. Regardless of the weak convergence of $V_\eps$ the transmission probability $|T_\eps(k,\alpha)|^2$
for each $\alpha\in \Lambda_\Phi$ converges to the value
\begin{equation}\label{T2forDeltaPrime}
|T(k,\alpha)|^2=\frac{4}{(\theta_\alpha^{-1} +\theta_\alpha)^2}
\end{equation}
that  does not depend  on energy of particles (cf. \cite[Sec.~5]{GolovatyHryniv:2010}). The limit operator $S(\alpha, 0)$ corresponds to the point interaction at the origin with the coupling matrix given by \eqref{DeltaPrimeMatrix}, $\theta=\theta_\alpha$.
As opposed to the case of the $\delta$-potential, this result is shape-dependent via  the resonant map $\theta$.

Our analysis of the exactly solvable models with piecewise-constant potentials  and
the computer simulation of more complicated models may suggest that
if $\varepsilon^{-2}\Phi(\varepsilon^{-1}\cdot)\to \delta'(x)$ as $\eps\to 0$,
then $|\theta_\alpha|\to +\infty$ as $\alpha\to +\infty$ and $|\theta_\alpha|\to 0$
as $\alpha\to -\infty$. Therefore the transmission probability for the $\delta'$-like potentials given by \eqref{T2forDeltaPrime} is very small for large $|\alpha|$.
As we will see in the next special case, it is not true in general.

It is noteworthy that the operator $S(\alpha, 0)$ can also appear as a solvable model for
the Schr\"{o}dinger operator  with the potential $V_\eps$ when $\beta$ is different from $0$.
Indeed, for some resonant values of $\alpha$ it is possible for potentials $\Phi$ and $\Psi$ to be
uncoupled, i.e., $\kappa_\alpha=0$. If, for instance, a potential $\Psi$ is sign-changing,
then it is possible for the integral $\int_{-1}^1\Psi u_\alpha^2\,ds$ to be zero.

\subsubsection{Potentials with total transparency at resonances}
An interesting case occurs when for a pair of potentials $\Phi$ and $\Psi$ the intercoupling map
$\kappa$ is identically zero and the resonant map $\theta$ is unimodular, $|\theta|=1$.
Then the marginal transmission probability across  the potential $V_\eps$ is given by
\begin{equation*}
    |T(k,0)|^2=\begin{cases}
                    1 & \hbox{ if $\alpha\in\Lambda_\Phi$,} \\
                       0 & \hbox{ otherwise.}
                     \end{cases}
\end{equation*}
Hence, either the potential is asymptotically  opaque for particles or else asymptotically totally transparent at resonances.
For instance, this kind of case arises when $\Phi$  is an even function, whereas  $\Psi$ is an odd one.
Then each eigenfunction $u_\alpha$ of \eqref{NeumanProblemWithAlpha} is either odd or even. In any case,  $|\theta_\alpha|=1$ and $\kappa_\alpha=\theta_\alpha^{-1} \int_{-1}^1\Psi u_\alpha^2\,ds=0$, since  the square of $u_\alpha$ is an even function.

%%%%%%%%%%%%%%%%%%%%%%%%%

\end{document}